\documentclass[a4paper]{amsart}%
\usepackage{amsmath}
\usepackage{amsfonts}%
\usepackage{amssymb}%
\usepackage{graphicx}
\providecommand{\U}[1]{\protect\rule{.1in}{.1in}}
\newtheorem{theorem}{Theorem}
\theoremstyle{plain}
\newtheorem{acknowledgement}{Acknowledgement}

\newtheorem{corollary}{Corollary}

\newtheorem{lemma}{Lemma}

\newtheorem{remark}{Remark}

\numberwithin{equation}{section}
\begin{document}
\title[Local and nonlocal optimal control in the source]{Local and nonlocal optimal control in the source}
\author{Julio Mu\~{n}oz}
\address{Departamento de Matem\'{a}ticas. Universidad de Castilla-La Mancha. Toledo, Spain.}
\email{Julio.Munoz@uclm.es}
\date{Mars 12, 2020}

\begin{abstract}
The analysis of an optimal control problem of nonlocal type is analyzed. The
results obtained are applied to the study the corresponding local optimal
control problems. The state equations are governed by p-laplacian elliptic
operators, of local and nonlocal type, and the costs belong to a wide class of
integral functionals. The nonlocal problem is formulated by means of a
convolution of the states with a kernel. This kernel depends on a parameter,
called horizon which, is responsible for the nonlocality of the equation. The
input function is the source of the elliptic equation. Existence of nonlocal
controls is obtained and a $G$-convergence result is employed in this task.
The limit of the solutions of the nonlocal optimal control, when the horizon
tends to zero, is analyzed and compared to the solution of the underlying
local optimal control problem.

\end{abstract}
\maketitle

\section{ Introduction\label{S1}}

The nonlocal models have showed a great level of capability in the study of
phenomena in many of the branches of science. They have been one of the main
alternatives to reformulate different types of problems in Applied
Mathematics. The usage of these models has been notable in fields like kinetic
equations, phase transitions, diffusion models and other themes of continuum
mechanics \cite{Sh, Metzler, Carreras, Metzler2, Gardiner, Bakunin, Neuman,
Bucur, Vazquez}. There are several ways to introduce the nonlocality when we
try to model some classical problems. Among others works we must highlight
\cite{Rossi2, Vazquez2, Gunzburger-Du-L-1, Du, kava} in the nonlocal framework
and \cite{caputo, Miller, rubin, Debnath, hilfer, Oustaloup, kibas, kulish}
from the point of view of the fractional analysis. In a general context, the
main idea to built a nonlocal model basically relies on considering
derivatives of nonlocal type, or of fractional derivatives, instead the
classical ones. This new way to measure the variability, somehow, allows to
introduce and modulate long-range interactions.

In our specific context, of optimal control problems governed by partial
differential equations, instead of considering differential equations, we
shall present a nonlocal model built by means integral equations. These
integrals are somehow, the convolution of the states with certain family of
kernels. This family is parametrized by a number, called \textit{horizon},
which is the responsible of the degree of the nonlocal interaction. The
proposed optimization problem is driven by the nonlocal $p$-laplacian as state
equation, and Dirichlet boundary conditions are imposed. The control is the
right-hand side forcing function, the source, and the cost to minimize belongs
to a fairly general class of integral functional.

The purpose of the present article is the analysis of this type of nonlocal
optimal control problem, the existence of solutions and their asymptotic
behavior when the nonlocality, the \textit{horizon}, tends to zero. Since in
the limit we recover the formulation of certain classical control problems,
some meaningful conclusions about approximation or existence of classical
solutions are obtained as well. Consequently, two different problems will be
addressed in the article, the nonlocal model and the classical or local
counterpart. To go into the details, we firstly specify the ambient space we
work on, and then, we shall formulate these two optimal control problems.

\subsection{Hypotheses}

Specifically, the framework in which we shall work can be described as
follows. The domain is $\Omega\subset\mathbb{R}^{N},$ a bounded open domain.
We define its extension $\Omega_{\delta}=\Omega\cup_{p\in\partial\Omega
}B\left(  p,\delta\right)  ,\ $where $B\left(  x,r\right)  $ is the notation
of an open ball centered at $x\in\mathbb{R}^{N}$ and radius $r>0$ and $\delta$
is a positive number.

About the right term of the elliptic equations, the function $f,$ called the
source, we assume $f\in L^{p^{\prime}}\left(  \Omega\right)  $ where
$p^{\prime}=\frac{p}{p-1}$ and $p>1.$ Concerning the kernels $\left(
k_{\delta}\right)  _{\delta>0},$ we assume that it is a sequence of
nonnegative radial functions such that for any $\delta,$%
\begin{equation}
\operatorname*{supp}k_{\delta}\subset B\left(  0,\delta\right) \label{1}%
\end{equation}
and%
\[
\frac{1}{C_{N}}\int_{B\left(  0,\delta\right)  }k_{\delta}\left(  \left\vert
z\right\vert \right)  dz=1
\]
where $C_{N}=\frac{1}{\operatorname*{meas}\left(  S^{N-1}\right)  }%
\int_{S^{N-1}}\left\vert \omega\cdot e\right\vert ^{p}d\sigma^{N-1}\left(
\omega\right)  ,$ where $\sigma^{N-1}$ stands for the $N-1$ dimensional
Haussdorff measure on the unit sphere $S^{N-1}$ and $e$ is any unitary vector
in $\mathbb{R}^{N}.$ In addition, the kernels satisfy the uniform estimation%
\begin{equation}
k_{\delta}\left(  \left\vert z\right\vert \right)  \geq\frac{c_{0}}{\left\vert
z\right\vert ^{N+\left(  s-1\right)  p}}\label{2}%
\end{equation}
where $c_{0}>0$ and $s\in\left(  0,1\right)  $ are given constants such that
$N>ps.$

The natural frame in which we shall work is the nonlocal energy space
\begin{equation}
X=\left\{  u\in L^{p}\left(  \Omega_{\delta}\right)  :B\left(  u,u\right)
<\infty\right\} \label{3}%
\end{equation}
where $B$ is the operator defined in $X\times X$ by means of the formula%
\begin{equation}
B\left(  u,v\right)  =\int_{\Omega_{\delta}}\int_{\Omega_{\delta}}k_{\delta
}\left(  \left\vert x^{\prime}-x\right\vert \right)  \frac{\left\vert u\left(
x^{\prime}\right)  -u\left(  x\right)  \right\vert ^{p-2}\left(  u\left(
x^{\prime}\right)  -u\left(  x\right)  \right)  }{\left\vert x^{\prime
}-x\right\vert ^{p}}\left(  v\left(  x^{\prime}\right)  -v\left(  x\right)
\right)  dx^{\prime}dx.\label{4}%
\end{equation}
We define also the constrained energy space as%
\[
X_{0}=\left\{  u\in X:u=0\text{ in }\Omega_{\delta}\setminus\Omega\right\}
\]
It is well-known that for any given $\delta>0$ the space $X=X\left(
\delta\right)  $ is a Banach space with the norm
\[
\left\Vert u\right\Vert _{X}=\left\Vert u\right\Vert _{L^{p}\left(
\Omega_{\delta}\right)  }+\left(  B\left(  u,u\right)  \right)  ^{1/p}.
\]
The dual of $X$ will be denoted by $X^{\prime}\ $and can be endowed with the
norm defined by
\[
\left\Vert g\right\Vert _{X^{\prime}}=\sup\left\{  \left\langle
g,w\right\rangle _{X^{\prime}\times X}:w\in X,\text{ }\left\Vert w\right\Vert
_{X}=1\right\}  .
\]
Analogous definitions applies to the space $X_{0}=X_{0}\left(  \delta\right)
.$

There is another functional space that we will use in the formulation of our
problem and that is susceptible to be used as a set of controls. It is the
space of diffusion coefficients, that is
\[
\mathcal{H}\doteq\left\{  h:\Omega_{\delta}\rightarrow\mathbb{R}\mid h\left(
x\right)  \in\lbrack h_{\min},h_{\max}]\text{ a.e. }x\in\Omega,\text{
}h=0\text{ in }\Omega_{\delta}-\Omega\right\}  ,
\]
where $h_{\min}\ $and $h_{\max}$ are positive constants such that $0<h_{\min
}<h_{\max}.$

\subsection{Formulation of the problems}

\subsubsection{Nonlocal Optimal control in the source.\textbf{\ }}

The nonlocal optimal control in the source is an optimal control problem
denoted by $\left(  \mathcal{P}^{\delta}\right)  $ whose formulation is as
follows: the problem, for each $\delta>0$ fixed, consists on finding $g\in
L^{p^{\prime}}\left(  \Omega\right)  $ such that minimizes the functional%
\begin{equation}
I_{\delta}\left(  g,u\right)  =\int_{\Omega}F\left(  x,u\left(  x\right)
,g\left(  x\right)  \right)  dx,\label{a}%
\end{equation}
where $u\in X$ is the solution of the nonlocal boundary problem $\left(
P^{\delta}\right)  $%
\begin{equation}
\left\{
\begin{tabular}
[c]{l}%
$\displaystyle B_{h}\left(  u,w\right)  =\int_{\Omega}g\left(  x\right)
w\left(  x\right)  dx,\text{ for any }w\in X_{0},\smallskip$\\
$\displaystyle u=u_{0}\text{ in }\Omega_{\delta}\setminus\Omega,$%
\end{tabular}
\right.  \ \ \ \ \ \label{b}%
\end{equation}
where%
\begin{equation}
B_{h}\left(  u,w\right)  =\int_{\Omega_{\delta}}\int_{\Omega_{\delta}}H\left(
x^{\prime},x\right)  k_{\delta}\left(  \left\vert x^{\prime}-x\right\vert
\right)  \frac{\left\vert u\left(  x^{\prime}\right)  -u\left(  x\right)
\right\vert ^{p-2}\left(  u\left(  x^{\prime}\right)  -u\left(  x\right)
\right)  }{\left\vert x^{\prime}-x\right\vert ^{p}}\left(  v\left(  x^{\prime
}\right)  -v\left(  x\right)  \right)  dx^{\prime}dx\label{bb}%
\end{equation}
$H\left(  x^{\prime},x\right)  \doteq\frac{h\left(  x^{\prime}\right)
+h\left(  x\right)  }{2}$ and $u_{0}$ is a given function. The above nonlocal
boundary condition, $u=u_{0}$ in $\Omega_{\delta}\setminus\Omega,$ must be
interpreted in the sense of traces. Indeed, in order to make sense it is
necessary that $u_{0}\ $belongs to the space $\widetilde{X}_{0}=\left\{
\left.  w\right\vert _{\Omega_{\delta}\setminus\Omega}:w\in X\right\}  $. This
space is well defined independently of the parameter $s\in\left(  0,1\right)
$ we choose in (\ref{2}). It is easy to check that a norm for this space is
the one defined as
\[
\left\Vert v\right\Vert _{\widetilde{X}_{0}}=\inf\left\{  \left\Vert
w\right\Vert _{L^{p}\left(  \Omega_{\delta}\right)  }+\left(  B\left(
w,w\right)  \right)  ^{1/p}:w\in X\text{ such that }\left.  w\right\vert
_{\Omega_{\delta}\setminus\Omega}=v\right\}  .
\]

The integrand $F$ that appear in the cost we want to optimize, is under the
format%
\begin{equation}
F\left(  x,u\left(  x\right)  ,g\left(  x\right)  \right)  =G\left(
x,u\left(  x\right)  \right)  +\beta\left\vert g\left(  x\right)  \right\vert
^{p^{\prime}}+\gamma B_{h_{0}}\left(  u,u\right)  ,\label{integrand_source}%
\end{equation}
where $\beta$ and $\gamma$ are given positive constants, $h_{0}\in\mathcal{H}$
is also given, and $G:\mathbb{R}\times\mathbb{R\rightarrow R}$ is assumed to
be a measurable positive function such that $G(x,\cdot)$ is uniformly
Lipschitz continuous, that is, for any $x\in\Omega$ and any $\left(
u,v\right)  \in\mathbb{R}^{2}$ there exists a positive constant $L$ such that
$\left\vert G\left(  x,u\right)  -G\left(  x,v\right)  \right\vert \leq
L\left\vert u-v\right\vert .$

We formulate the nonlocal optimal control problem as%
\begin{equation}
\min_{\left(  g,u\right)  \in\mathcal{A}^{\delta}}I_{\delta}\left(  g,u\right)
\label{NLCP}%
\end{equation}
where%
\[
\mathcal{A}^{\delta}\doteq\left\{  \left(  f,v\right)  \in L^{p^{\prime}%
}\left(  \Omega\right)  \times X:v\text{ solves (\ref{b}) with }g=f\right\}  .
\]
Recall that under the above circumstances, the nonlocal participating states
in (\ref{b}), can be viewed as elements of the convex space
\[
u_{0}+X_{0}\doteq\left\{  v\in X:v=u_{0}+w\text{ where }w\in X_{0}\right\}  .
\]

\subsubsection{The local optimal control in the source}

The corresponding local optimal control is a problem denoted by $\left(
\mathcal{P}^{loc}\right)  $ whose goal is to find $g\in L^{p^{\prime}}\left(
\Omega\right)  $ such that minimizes the functional%
\begin{equation}
I\left(  g,u\right)  =\int_{\Omega}\left(  G\left(  x,u\left(  x\right)
\right)  +\beta\left\vert g\left(  x\right)  \right\vert ^{p^{\prime}}\right)
dx+\gamma b_{h_{0}}\left(  u,u\right)  ,\label{c}%
\end{equation}
where $u\in W^{1,p}\left(  \Omega\right)  $ is the solution of the local
boundary problem $\left(  P^{loc}\right)  $%
\begin{equation}
\left\{
\begin{tabular}
[c]{l}%
$\displaystyle b_{h}\left(  u,w\right)  =\int_{\Omega}g\left(  x\right)
w\left(  x\right)  dx,\text{ for any }w\in W_{0}^{1,p}\left(  \Omega\right)
\smallskip$\\
$\displaystyle u=u_{0}\text{ in }\partial\Omega,$%
\end{tabular}
\right.  \ \ \ \label{d}%
\end{equation}
$b\left(  \cdot,\cdot\right)  $ is the operator defined in $W^{1,p}\left(
\Omega\right)  \times W^{1,p}\left(  \Omega\right)  $ by means of%
\[
b_{h}\left(  u,v\right)  \doteq\int_{\Omega}h\left(  x\right)  \left\vert
\nabla u\left(  x\right)  \right\vert ^{p-2}\nabla u\left(  x\right)  \nabla
v\left(  x\right)  dx,
\]
$h_{0}\in\mathcal{H}$ and $u_{0}$ is a given function from the trace
fractional Sobolev space $W^{1-1/p,p}\left(  \partial\Omega\right)  $.
Throughout the article, when dealing with the local case, we shall assume
$\Omega$ is a bounded smooth domain.

The statement of the local optimal control problem is%
\begin{equation}
\min_{\left(  g,u\right)  \in\mathcal{A}^{loc}}I\left(  g,u\right) \label{e}%
\end{equation}
where%
\[
\mathcal{A}^{loc}\doteq\left\{  \left(  f,v\right)  \in L^{p^{\prime}}\left(
\Omega\right)  \times W^{1,p}\left(  \Omega\right)  :v\text{ solves (\ref{d})
with }g=f\right\}  .
\]
As usual, if we identify $u_{0}$ with a function $V_{0}\in W^{1,p}\left(
\Omega\right)  $ whose trace is $u_{0},$ and at the same time, $V_{0}$ is
denoted by $u_{0}$ too, then, as usual, the competing states are those that
form space
\[
u_{0}+W_{0}^{1,p}\left(  \Omega\right)  \doteq\left\{  v\in W^{1,p}\left(
\Omega\right)  :v=u_{0}+w\text{ where }w\in W_{0}^{1,p}\left(  \Omega\right)
\right\}  .
\]

The analysis of this type of problems is a subject that has been extensively
studied in previous works \cite{Gunzburger-Du-L-1, Ak, Hinds-Radu,
D'Elia-Gunz2, D'Elia-Gunz2b}. As far as the author knows, the first work
dealing with nonlocal optimal control problems is \cite{D'Elia-Gunz1}. A
series of articles containing different type of controls have appeared in the
last years. Some good samples are \cite{D'Elia-Gunz1, D'Elia-Gunz2,
FernandezBonder, Warma}. About the analysis of $G$-convergence or $\Gamma
$-convergence the reader can consult \cite{Ponce2, Zhou, D'Elia-Gunz1,
Mengesha2, bellidob, Mengesha}. We can find some theoretical advances about
the explicit computation of the limit problem. In this sense we must underline
among others \cite{Mengesha2, Bonder, Bellido-Egrafov, waurick, Du}. Much more
should be commented about the influence that this type of problems has
received from an outstanding list of seminal papers whose main topic, has been
the analysis and characterization of Sobolev Spaces. See for instance
\cite{Bourgain-Brezis, Ponce2, Rossi2, Rossi3, Mazon,
DiNezza-Palatucci-Valdinoci, Du}. In what concerns the numerical analysis of
nonlocal problems see \cite{DeliaNumer} and references therein.

We must say that to a great extent, the work \cite{D'Elia-Gunz1} has served as
inspiration for the present article. Nonetheless, we must emphasize the
techniques we use here, in some aspects, substantially differ from the ones
employed there. One of the features of our development is the usage of a
principle of minimum energy in the in order to characterize the $G$%
-convergence of the state equation (see \cite[Chapter 5, p. 162]{Jikov} for a
detailed study in a concrete linear case). Recall that since we are dealing
with the exponent $p>1,$ the linearity for the $p$-laplacian disappears and
consequently, the classical Lax-Milgram Theorem no longer applies. Besides, in
\cite{D'Elia-Gunz1} this linearity and the specificity of the type of cost
functionals, jointly with the necessary conditions of optimality are the key
points for the achievement of existence of optimal controls. By contrast, in
our context, the proof of existence, both for the state equation and the
optimal control problem, is obtained by means of the Direct Method and the
result of $G$-convergence. After, we prove convergence of the nonlocal state
equation and the nonlocal optimal control problem to the local ones. Even
though these achievements could be significant since the analysis could be
applied to a rather general class of cost functionals, the results obtained
for the particular case $p=2$ are not less attractive. The reason is that for
such a case the non-local model can approximate classical problems including
the squared gradient within the cost functional. Though we have not examined
any numerical method for the approximation of solutions yet, some techniques
derived from a maximum principle (see \cite{Cea-Mala} for the local case)
could be explored in order to build a descent method for the case $p=2$ (see
\cite{Andres}).

\subsection{Results and organization}

The purpose of this manuscript is twofold: there is a first part of the paper
devoted to study the existence of nonlocal optimal designs. This objective is
achieved for a cost functional class whose format may include the non-linear
term of the non-local operator. See Theorem \ref{Well-posedness} in Section
\ref{S3}. The proof of this theorem is basically, based on a previous result
of $G$- convergence (Theorem \ref{G-convergence} in Section \ref{S3}). The aim
of the second part is the convergence of the nonlocal problem toward the local
optimal design one. In a first stage we prove convergence of the state
equation to the classical $p$-laplacian when the \textit{horizon}
$\delta\rightarrow0$ (Theorem \ref{Th2} in Section \ref{S4}). Then, we face
the study of convergence for the optimal design problem. The main result is
Theorem \ref{Th3}, in Section \ref{S5}. A case of particular interest, the one
that we assume $p=2,$ is analyzed. The type of cost functional for which we
study the convergence, includes the nonlocal gradient, and consequently, the
local counterpart optimal problem we approximate contains the square of the
gradient (Theorem \ref{Th4} in Section \ref{S5}). In order to facilitate the
reading of the article, some specific preliminary results are previously
explained in Section \ref{S2}. Some compactness and basic inequalities are
commented, and the proof of existence of solution for the nonlocal state
problem is analyzed (Theorem \ref{Th1}).

\section{Preliminary results and well-posedness of the state
equation\label{S2}}

\subsection{Preliminaries\label{S2_1}}

Here we review some technical tools we are going to use.

\begin{enumerate}
\item \label{S2_1_1}The embedding%
\[
X_{0}\subset L^{p}\left(  \Omega\right)
\]
is compact. In order to check that we firstly notice $X_{0}\subset
W^{s,p}\left(  \Omega_{\delta}\right)  ,$ and since the elements of $X_{0}$
vanish in $\Omega_{\delta}\setminus\Omega,$ then extension by zero outside
$\Omega_{\delta}$ gives rise to elements of $W^{s,p}\left(  \mathbb{R}%
^{N}\right)  $ (see \cite[Lemma 5.1]{DiNezza-Palatucci-Valdinoci}). Then
\[
X_{0}\subset W_{0}^{s,p}\left(  \Omega\right)  =\left\{  f\in W^{s,p}\left(
\mathbb{R}^{N}\right)  :f=0\text{ in }\mathbb{R}^{N}\setminus\Omega\right\}
\]
Besides, we are in position to state the existence of a constant $c=c\left(
N,s,p\right)  $ such that for any $w\in X_{0}$%
\begin{equation}
c\left\Vert w\right\Vert _{L^{p}\left(  \Omega_{\delta}\right)  }^{p}\leq
\int_{\Omega_{\delta}}\int_{\Omega_{\delta}}\frac{\left\vert w\left(
x^{\prime}\right)  -w\left(  x\right)  \right\vert ^{p}}{\left\vert x^{\prime
}-x\right\vert ^{N+sp}}dx^{\prime}dx\label{Pre0}%
\end{equation}
(see \cite[Th. 6.5]{DiNezza-Palatucci-Valdinoci}). By paying attention to the
hypotheses on the kernel (\ref{2}), and using (\ref{Pre0}) we conclude there
is a positive constant $C$ such that the nonlocal Poincar\'{e} inequality
\begin{equation}
C\left\Vert w\right\Vert _{L^{p}\left(  \Omega_{\delta}\right)  }^{p}\leq
B_{h}\left(  w,w\right)  .\label{Prel3}%
\end{equation}
holds for any $w\in X_{0}.$\newline We consider now a sequence $\left(
w_{j}\right)  _{j}\in X_{0}$ uniformly bounded in $X_{0},$ that is, there is a
constant $C$ such that for every $j$
\begin{equation}
B_{h}\left(  w_{j},w_{j}\right)  \leq C.\label{Prel1}%
\end{equation}
By (\ref{Prel3}) $\left(  w_{j}\right)  _{j}$ is uniformly bounded
$L^{p}\left(  \Omega_{\delta}\right)  \ $which, jointly with (\ref{2})
guarantees $\left(  w_{j}\right)  _{j}$ is uniformly bounded in $W_{0}%
^{s,p}\left(  \Omega_{\delta}\right)  $. We employ now the compact embedding
$W_{0}^{s,p}\left(  \Omega_{\delta}\right)  \subset L^{p}\left(
\Omega_{\delta}\right)  $ (see \cite[Th. 7.1]{DiNezza-Palatucci-Valdinoci}) to
ensure the existence of a subsequence from $\left(  w_{j}\right)  _{j},$ still
denoted by $\left(  w_{j}\right)  _{j},$ such that $w_{j}\rightarrow w$
strongly in $L^{p}\left(  \Omega_{\delta}\right)  ,$ for some $w\in X_{0}.$

\item \label{S2_1_2}If we take a sequence $\left(  w_{j}\right)  _{j}$ from
$u_{0}+X_{0}$ such that $B_{h}\left(  w_{j},w_{j}\right)  \leq C,$ then
$u_{j}-u_{0}\in X_{0}.$ But, since%
\[
B_{h}\left(  w_{j}-u_{0},w_{j}-u_{0}\right)  \leq c\left(  B_{h}\left(
w_{j},w_{j}\right)  +B_{h}\left(  u_{0},u_{0}\right)  \right)
\]
for a certain constant $c,$ then%
\[
B_{h}\left(  w_{j}-u_{0},w_{j}-u_{0}\right)  \leq C
\]
If we apply now the nonlocal Poincar\'{e} inequality (\ref{Prel3}) we have%
\begin{equation}
C\left\Vert w_{j}-u_{0}\right\Vert _{L^{p}}^{p}\leq B_{h}\left(  w_{j}%
-u_{0},w_{j}-u_{0}\right)  ,\label{Prel2}%
\end{equation}
whereby we state the sequence $\left(  w_{j}\right)  _{j}$ is uniformly
bounded in $L^{p}\left(  \Omega_{\delta}\right)  $ and therefore, there exists
a function $w\in L^{p}\left(  \Omega\right)  $ such that, for a subsequence of
$\left(  w_{j}\right)  _{j},$ still denoted by $\left(  w_{j}\right)  _{j},$
$w_{j}\rightarrow u$ strongly in $L^{p}\left(  \Omega_{\delta}\right)  .$
Moreover, since $u_{0}+X_{0}$ is a closed set, $w\in u_{0}+X_{0}.$

\item \label{S2_1_3}Let $\left(  g_{\delta},u_{\delta}\right)  _{\delta}$ be a
sequence of pairs such that the uniform estimation%
\[
B_{h}\left(  u_{\delta},u_{\delta}\right)  \leq C,
\]
is fulfilled (where $C$ is a positive constant). Then, from $\left(
u_{\delta}\right)  _{\delta}$ we can extract a subsequence, labelled also by
$u_{\delta},$ such that $u_{\delta}\rightarrow u$ strongly in $L^{p}\left(
\Omega\right)  $ and $u\in W^{1,p}\left(  \Omega\right)  $ (see \cite[Th.
1.2]{Ponce1}). Furthermore, the following inequality is fulfilled%
\begin{equation}
\lim_{\delta\rightarrow0}B_{h}\left(  u_{\delta},u_{\delta}\right)  \geq
\int_{\Omega}h\left(  x\right)  \left\vert \nabla u\left(  x\right)
\right\vert ^{p}dx\label{Prel4}%
\end{equation}
(see \cite{Ponce1, Andres-Julio3, Munoz}). Besides, it is also well-known that
if $u_{\delta}=u\in W^{1,p}\left(  \Omega\right)  ,$ then the above limit is
\begin{equation}
\lim_{\delta\rightarrow0}B_{h}\left(  u,u\right)  =\int_{\Omega}h\left(
x\right)  \left\vert \nabla u\left(  x\right)  \right\vert ^{p}dx\label{Prel5}%
\end{equation}
(see \cite[Corollary 1]{Bourgain-Brezis} and \cite[Th. 8]{Andres-Julio2}).
\end{enumerate}

\subsection{The state equation}

For the well-posedness of the nonlocal control problem $\left(  \mathcal{P}%
_{s}^{\delta}\right)  $ it is imperative to prove existence and uniqueness for
the nonlocal boundary problem $\left(  P^{\delta}\right)  $. A remarkable fact
that will be employed for this goal is the characterization of (\ref{1}) by
means of a Dirichlet principle. For the proof, we just need to adapt (because
we have to include the nonlocal boundary condition $u_{0})$ the lines given in
\cite{Bonder}.\smallskip

Throughout this section, $u_{0}\in\widetilde{X}_{0},$ $\delta>0$ and $g$ $\in
L^{p^{\prime}}\left(  \Omega\right)  $ are assumed to be fixed. We seek a
solution to the problem (\ref{b}) and as we have commented, the crucial point
in this searching is the inherent relation of the nonlocal boundary problem
with the following minimization problem:%
\begin{equation}
\min_{w\in u_{0}+X_{0}}\mathcal{J}\left(  w\right)  \label{NL_Dirichlet}%
\end{equation}
where%
\[
\mathcal{J}\left(  w\right)  \doteq\frac{1}{p}B_{h}\left(  w,w\right)
-\int_{\Omega}g\left(  x\right)  w\left(  x\right)  dx.
\]

\begin{lemma}
There exists a solution $u\in u_{0}+X_{0},$ to the problem of minimization
(\ref{NL_Dirichlet}).
\end{lemma}

\begin{proof}
First of all, we check $\mathcal{J}$ is bounded from below. Let $w$ be any
function from $X$ such that $w-u_{0}\in X_{0}$. By using the nonlocal
Poincar\'{e} inequality (\ref{Prel2}) there is a constant $c>o$ such that
\[
c\left\Vert w-u_{0}\right\Vert _{L^{p}}^{p}\leq B_{h}\left(  w,w\right)
+B_{h}\left(  u_{0},u_{0}\right)  ,
\]
whence we have%
\begin{equation}
\left\Vert w\right\Vert _{L^{p}}-\left\Vert u_{0}\right\Vert _{L^{p}}%
\leq\left\Vert w-u_{0}\right\Vert _{L^{p}}\leq\left(  \frac{B_{h}\left(
w,w\right)  +B_{h}\left(  u_{0},u_{0}\right)  }{c}\right)  ^{1/p}%
.\label{basic_ineq}%
\end{equation}
If we apply now the H\"{o}lder's inequality and Young's inequality we get
\begin{align*}
\mathcal{J}\left(  w\right)    & \geq\frac{1}{p}B_{h}\left(  w,w\right)
-\left\Vert g\right\Vert _{L^{p^{\prime}}}\left\Vert w\right\Vert _{L^{p}}\\
& \geq\frac{1}{p}B_{h}\left(  w,w\right)  -\left\Vert g\right\Vert
_{L^{p^{\prime}}}\frac{1}{c^{1/p}}\left(  B_{h}\left(  w,w\right)
+B_{h}\left(  u_{0},u_{0}\right)  \right)  ^{1/p}-\left\Vert g\right\Vert
_{L^{p^{\prime}}}\left\Vert u_{0}\right\Vert _{L^{p}}\\
& \geq\frac{1}{p}B_{h}\left(  w,w\right)  -\frac{1}{p}\left(  B_{h}\left(
w,w\right)  +B_{h}\left(  u_{0},u_{0}\right)  \right)  -\frac{1}{p^{\prime}%
}\left(  \frac{\left\Vert g\right\Vert _{L^{p^{\prime}}}}{c^{1/p}}\right)
^{p^{\prime}}-\left\Vert g\right\Vert _{L^{p^{\prime}}}\left\Vert
u_{0}\right\Vert _{L^{p}}\\
& =-\frac{1}{p}B_{h}\left(  u_{0},u_{0}\right)  -\frac{1}{p^{\prime}}\left(
\frac{\left\Vert g\right\Vert _{L^{p^{\prime}}}}{c^{1/p}}\right)  ^{p^{\prime
}}-\left\Vert g\right\Vert _{L^{p^{\prime}}}\left\Vert u_{0}\right\Vert
_{L^{p}}.
\end{align*}
To prove the existence of solution we take a minimizing sequence $\left(
u_{j}\right)  \subset u_{0}+X_{0}$ so that
\begin{equation}
m=\lim_{j\rightarrow\infty}\left(  \frac{1}{p}B_{h}\left(  u_{j},u_{j}\right)
-\int_{\Omega}g\left(  x\right)  u_{j}\left(  x\right)  dx\right)
\label{inf_1}%
\end{equation}
where $m$ is the infimum $\inf_{w\in u_{0}+X_{0}}\mathcal{J}\left(  w\right)
.$ From this convergence we ensure that there is a constant $C>0$ such that
\[
\frac{1}{p}B_{h}\left(  u_{j},u_{j}\right)  -\int g\left(  x\right)
u_{j}\left(  x\right)  dx\leq C
\]
for any $j.$ Thus, we get the estimation%
\begin{align*}
0 &  \leq B_{h}\left(  u_{j},u_{j}\right)  \leq C+\left\vert \int_{\Omega
}g\left(  x\right)  u_{j}\left(  x\right)  dx\right\vert \\
&  \leq C\left(  1+\left\Vert g\right\Vert _{L^{2}\left(  \Omega\right)
}\left\Vert u_{j}\right\Vert _{L^{2}\left(  \Omega\right)  }\right)
\end{align*}
with $C>0.$ Again, the nonlocal Poincar\'{e} inequality gives
\[
c\left\Vert u_{j}-u_{0}\right\Vert _{L^{p}}^{p}\leq B_{h}\left(  u_{j}%
,u_{j}\right)  +B_{h}\left(  u_{0},u_{0}\right)  \leq C\left(  1+\left\Vert
g\right\Vert _{L^{p^{\prime}}\left(  \Omega\right)  }\left\Vert u_{j}%
\right\Vert _{L^{p}\left(  \Omega\right)  }\right)  +B_{h}\left(  u_{0}%
,u_{0}\right)
\]
and therefore%
\[
\left\Vert u_{j}\right\Vert _{L^{p}}^{p}\leq C\left(  1+\left\Vert
g\right\Vert _{L^{p^{\prime}}\left(  \Omega\right)  }\left\Vert u_{j}%
\right\Vert _{L^{p}\left(  \Omega\right)  }+\left\Vert u_{0}\right\Vert
_{L^{p}}^{p}\right)
\]
From the two above inequalities we deduce the sequences $B_{h}\left(
u_{j},u_{j}\right)  \ $and $\left\Vert u_{j}\right\Vert _{L^{p}}$ are
uniformly bounded. By virtue of the compactness embedding $X_{0}\subset
L^{p},$ we know there is a subsequence of $\left(  u_{j}\right)  ,$ which will
be denoted also by $\left(  u_{j}\right)  ,$ strongly convergent in
$L^{p}\left(  \Omega\right)  $ to some $u\in u_{0}+X_{0}.$\newline We retake
(\ref{inf_1}) and use the lower semicontinuity in $L^{p}$ of the operator
$\mathcal{J}$ to write%
\begin{align*}
m &  =\lim_{j}\frac{1}{p}B_{h}\left(  u_{j},u_{j}\right)  -\lim_{j}%
\int_{\Omega}g\left(  x\right)  u_{j}\left(  x\right)  dx\\
&  \geq\frac{1}{p}B_{h}\left(  u,u\right)  -\int_{\Omega}g\left(  x\right)
u\left(  x\right)  dx\\
&  =\mathcal{J}\left(  u\right)  .
\end{align*}
From this inequality we conclude that $u$ is a minimizer.
\end{proof}

\begin{lemma}
$u$ is a solution of the minimization principle (\ref{NL_Dirichlet}) if, and
only if, $u$ solves the problem (\ref{b}).
\end{lemma}

The proof is standard. Assume $u$ solves (\ref{b}). We have only note that if
we take any $v\in u_{0}+X_{0}$ then $w\doteq u-v\in X_{0}$ and $B_{h}\left(
u,w\right)  =\int_{\Omega}g\left(  x\right)  w\left(  x\right)  dx,$ that is%
\[
B_{h}\left(  u,u\right)  =B_{h}\left(  u,v\right)  +\int_{\Omega}g\left(
u-v\right)  dx.
\]
By applying Young's inequality to the first term of the right part in the
above equality we get%
\[
B_{h}\left(  u,u\right)  \leq\frac{1}{p}B_{h}\left(  v,v\right)  +\frac
{1}{p^{\prime}}B_{h}\left(  u,u\right)  +\int_{\Omega}g\left(  u-v\right)  dx,
\]
and thus%
\[
\frac{1}{p}B_{h}\left(  u,u\right)  -\int_{\Omega}gudx\leq\frac{1}{p}%
B_{h}\left(  v,v\right)  -\int_{\Omega}gvdx
\]
which is equivalent to write $\mathcal{J}\left(  u\right)  \leq\mathcal{J}%
\left(  v\right)  .$\newline And reciprocally, if $u$ is a minimizer of
$\mathcal{J}$ on $u_{0}+X_{0}$ then we can take the admissible function
$w=u+t\xi,$ where $\xi$ is any element form $X_{0}.$ Since the function
$j\left(  t\right)  =\mathcal{J}\left(  u+t\xi\right)  $ attains a minimum at
$t=0,$ then $j^{\prime}\left(  0\right)  =0$ and this equality can be easily
rewritten as $B_{h}\left(  u,\xi\right)  =\int_{\Omega}g\left(  x\right)
\xi\left(  x\right)  dx.$

\begin{theorem}
\label{Th1}There exists a unique solution to the nonlocal boundary problem
$\left(  P^{\delta}\right)  $ given in (\ref{b}).
\end{theorem}

All that remains is to prove the uniqueness. The proof is automatic due to the
convexity of $\mathcal{J}:$ indeed, if $u$ and $v$ are minimizers, then
\[
m=\min_{w\in u_{0}+X_{0}}\mathcal{J}\left(  w\right)  =J\left(  u\right)
=J\left(  v\right)  .
\]
Besides, for any $\alpha\in\left(  0,1\right)  ,$ the function $\alpha
u+\left(  1-\alpha\right)  v$ is admissible for minimization principle. Then,
thanks to the strict convexity of $\mathcal{J},$ we deduce that
\[
m\leq\mathcal{J}\left(  \alpha u+\left(  1-\alpha\right)  v\right)
<\alpha\mathcal{J}\left(  u\right)  +\left(  1-\alpha\right)  \mathcal{J}%
\left(  v\right)  =m
\]
which is a contradiction.

\begin{remark}
The result remains valid if we assume $g$ to be in the space $X_{0}^{\prime
}\ $and the proof follows along the same lines from above.
\end{remark}

The existence and uniqueness of solution for the local state equation $\left(
P^{loc}\right)  $ is a basic issue. Even we have to adapt some details,
\cite{Chipot} is a reference we can follow in order to carry out this task.
Although the details about the proof are interesting, the uniqueness is an
aspect that could be analyzed apart. Indeed, if $u$ and $v$ are two different
solutions of the state equation (\ref{d}), then $b_{h}\left(  u,w\right)
=b\left(  v,w\right)  $ for any $w\in X_{0}.$ This is to say that for any
$w\in X_{0}$%
\begin{equation}
\int_{\Omega}h\left(  x\right)  \left(  \left\vert \nabla u\left(  x\right)
\right\vert ^{p-2}\nabla u\left(  x\right)  \nabla w\left(  x\right)
-\left\vert \nabla v\left(  x\right)  \right\vert ^{p-2}\nabla v\left(
x\right)  \nabla w\left(  x\right)  \right)  dx=0.\label{uni}%
\end{equation}
By taking $w=u-v$ we obtain%
\[
\int_{\Omega}h\left(  x\right)  \left(  \left\vert \nabla u\left(  x\right)
\right\vert ^{p-2}\nabla u\left(  x\right)  -\left\vert \nabla v\left(
x\right)  \right\vert ^{p-2}\nabla v\left(  x\right)  \right)  \nabla\left(
u-v\right)  \left(  x\right)  dx=0.
\]
At this point we take into account the next elementary inequality: if
$1<p<\infty,$ then there exist two positive constants $C=C\left(  p\right)  $
and $c=c\left(  p\right)  $ such that for every $a,$ $b\in\mathbb{R}^{N}$%
\begin{equation}
c\left\{  \left\vert a\right\vert +\left\vert b\right\vert \right\}
^{p-2}\left\vert a-b\right\vert ^{2}\leq\left(  \left\vert a\right\vert
^{p-2}a-\left\vert b\right\vert ^{p-2}b\right)  \cdot\left(  a-b\right)  \leq
C\left\{  \left\vert a\right\vert +\left\vert b\right\vert \right\}
^{p-2}\left\vert a-b\right\vert ^{2}.\label{ele_ine}%
\end{equation}
Finally, by applying (\ref{ele_ine}) in (\ref{uni}) the uniqueness follows
(see \cite[Prop.17.3 and Th. 17.1]{Chipot}\textbf{).}

\section{$G$-Convergence for the state equation and existence of nonlocal
optimal controls\label{S3}}

Let $\left(  g_{j}\right)  _{j}$ be a minimizing sequences of controls for the
problem $\left(  \mathcal{P}^{\delta}\right)  $ and let $\left(  u_{j}\right)
_{j}$ be the corresponding sequence of states. As in the end the sequences we
are going to work with, are minimizing sequences, we shall assume that there
is a constant $C>0$ such that for any
\[
\int_{\Omega}\left\vert g_{j}\left(  x\right)  \right\vert ^{p^{\prime}}dx<C.
\]
Hence, we can extract a subsequence weakly convergent in $L^{p^{\prime}%
}\left(  \Omega\right)  $ to some $g\in L^{p^{\prime}}\left(  \Omega\right)
$. We also know the following variational equality for any $v\in X_{0}:$
\[
B_{h}\left(  u_{j},v\right)  =\int_{\Omega}g_{j}\left(  x\right)  v\left(
x\right)  dx
\]
In particular
\[
B_{h}\left(  u_{j},u_{j}-u_{0}\right)  =\int_{\Omega}g_{j}\left(  x\right)
\left(  u_{j}\left(  x\right)  -u_{0}\left(  x\right)  \right)  dx.
\]
Holder's inequality and the linearity of $B_{h}\left(  w,\cdot\right)  ,$ for
any $w\in X,$ lead us the estimation%
\[
B_{h}\left(  u_{j},u_{j}\right)  \leq\left\Vert g_{j}\right\Vert
_{L^{p^{\prime}}}\left(  \left\Vert u_{j}\right\Vert _{L^{p}}+\left\Vert
u_{0}\right\Vert _{L^{p}}\right)  +B_{h}\left(  u_{j},u_{0}\right)  .
\]
If we take into account (\ref{basic_ineq}) and make use of the Young's
inequality we deduce%
\begin{align*}
B_{h}\left(  u_{j},u_{j}\right)   & \leq\left\Vert g_{j}\right\Vert
_{L^{p^{\prime}}}\left(  \left(  \frac{B_{h}\left(  u_{j},u_{j}\right)
+B_{h}\left(  u_{0},u_{0}\right)  }{c}\right)  ^{1/p}+2\left\Vert
u_{0}\right\Vert _{L^{p}}\right) \\
& +\frac{1}{p^{\prime}}B_{h}\left(  u_{j},u_{j}\right)  +\frac{1}{p}%
B_{h}\left(  u_{0},u_{0}\right)  ,
\end{align*}
and thereby
\[
\left(  1-\frac{1}{p^{\prime}}\right)  B_{h}\left(  u_{j},u_{j}\right)  \leq
C+D\left(  B_{h}\left(  u_{j},u_{j}\right)  \right)  ^{1/p}%
\]
for some positive constants $C$ and $D.$ The above inequality implies
$B_{h}\left(  u_{j},u_{j}\right)  $ is uniformly bounded and by
(\ref{basic_ineq}) $\left\Vert u_{j}\right\Vert _{L^{p}}$ too. If at this
point we use point \ref{S2_1_2} from Subsection \ref{S2_1}, we can state the
strong convergence in $L^{p}$, at least for a subsequence of $\left(
u_{j}\right)  _{j},$ to some function $u^{\ast}\in u_{0}+X_{0}.$ Let $u$ be
the state associated to $g.$ We pose wether the identity $u=u^{\ast}$ is true
or not:

\begin{theorem}
[$G$-convergence]\label{G-convergence}Under the above circumstances%
\[
\lim_{j\rightarrow\infty}\min_{w\in u_{0}+X_{0}}\left\{  \frac{1}{p}%
B_{h}\left(  w,w\right)  -\int_{\Omega}g_{j}\left(  x\right)  w\left(
x\right)  dx\right\}  =\min_{w\in u_{0}+X_{0}}\left\{  \frac{1}{p}B_{h}\left(
w,w\right)  -\int_{\Omega}g\left(  x\right)  w\left(  x\right)  dx\right\}
\]
and $u=u^{\ast}.$
\end{theorem}

\begin{proof}
Assume $m_{j}$ and $m$ denote the minimum values from the left and right
respectively. We prove $\lim_{j}m_{j}\leq$ $m:$%
\begin{align*}
\lim_{j}m_{j} &  =\lim_{j}\left(  \frac{1}{p}B_{h}\left(  u_{j},u_{j}\right)
-\int_{\Omega}g_{j}\left(  x\right)  u_{j}\left(  x\right)  dx\right)  \\
&  \leq\lim_{j}\left(  \frac{1}{p}B_{h}\left(  u,u\right)  -\int_{\Omega}%
g_{j}\left(  x\right)  u\left(  x\right)  dx\right)  \\
&  =\frac{1}{p}B_{h}\left(  u,u\right)  -\int_{\Omega}g\left(  x\right)
u\left(  x\right)  dx\\
&  =\min_{w\in u_{0}+X_{0}}\left\{  \frac{1}{p}B_{h}\left(  w,w\right)
-\int_{\Omega}g\left(  x\right)  w\left(  x\right)  dx\right\}  .
\end{align*}
We check $\lim_{j}m_{j}\geq m:$ we know $u_{j}\rightarrow u^{\ast}$ strongly
in $L^{p},$ $g_{j}\rightharpoonup g$ weakly in $L^{p^{\prime}}$ and therefore
\[
\lim_{j}\int_{\Omega}g_{j}\left(  x\right)  u_{j}\left(  x\right)
dx=\int_{\Omega}g\left(  x\right)  u^{\ast}\left(  x\right)  dx.
\]
We apply these convergences to analyze the limit of the energy functional:
\begin{align*}
\lim_{j}m_{j} &  =\lim_{j}\left(  \frac{1}{p}B_{h}\left(  u_{j},u_{j}\right)
-\int_{\Omega}g_{j}\left(  x\right)  u_{j}\left(  x\right)  dx\right)  \\
&  =\frac{1}{p}\lim_{j}B_{h}\left(  u_{j},u_{j}\right)  -\int_{\Omega}g\left(
x\right)  u^{\ast}\left(  x\right)  dx\\
&  \geq\frac{1}{p}B_{h}\left(  u^{\ast},u^{\ast}\right)  -\int_{\Omega
}g\left(  x\right)  u^{\ast}\left(  x\right)  dx\\
&  \geq\frac{1}{p}B_{h}\left(  u,u\right)  -\int_{\Omega}g\left(  x\right)
u\left(  x\right)  dx\\
&  =\min_{w\in u_{0}+X_{0}}\left\{  \frac{1}{2}B_{h}\left(  w,w\right)
-\int_{\Omega}g\left(  x\right)  w\left(  x\right)  dx\right\}
\end{align*}
where the first inequality is due to the lower semicontinuity of the operator
$B_{h}\left(  \cdot,\cdot\right)  $ with respect to the weak convergence in
$L^{p}.$ We have proved $\lim_{j}m_{j}=$ $m.$ Also, from the above chain of
inequalities it is obvious to see that both $u$ and $u^{\ast}$ are solutions
to the problem (\ref{NL_Dirichlet}), then according to Theorem \ref{Th1}
$u=u^{\ast}.$
\end{proof}

\begin{corollary}
\label{remark_conver_norms}The following convergences hold%
\begin{equation}
\lim_{j\rightarrow\infty}B_{h}\left(  u_{j},u_{j}\right)  =B_{h}\left(
u,u\right)  ,\label{convergence_norms}%
\end{equation}
and%
\begin{equation}
\lim_{j\rightarrow\infty}B_{h}\left(  u_{j}-u,u_{j}-u\right)
=0.\label{convergence_norms_main}%
\end{equation}

\end{corollary}

\begin{proof}
(\ref{convergence_norms}) follows from the proof of the above theorem and can
be rewritten as this convergence of norms:%
\begin{align*}
&  \lim_{j\rightarrow\infty}\int_{\Omega_{\delta}}\int_{\Omega_{\delta}%
}H\left(  x^{\prime},x\right)  k_{\delta}\left(  \left\vert x^{\prime
}-x\right\vert \right)  \frac{\left\vert u_{j}\left(  x^{\prime}\right)
-u_{j}\left(  x\right)  \right\vert ^{p}}{\left\vert x^{\prime}-x\right\vert
^{p}}dx^{\prime}dx\\
&  =\int_{\Omega_{\delta}}\int_{\Omega_{\delta}}H\left(  x^{\prime},x\right)
k_{\delta}\left(  \left\vert x^{\prime}-x\right\vert \right)  \frac{\left\vert
u\left(  x^{\prime}\right)  -u\left(  x\right)  \right\vert ^{p}}{\left\vert
x^{\prime}-x\right\vert ^{p}}dx^{\prime}dx.
\end{align*}
But this convergence is equivalent to say that the norm of the sequence
\[
\Psi_{j}\left(  x^{\prime},x\right)  =H^{1/p}\left(  x^{\prime},x\right)
k_{\delta}^{1/p}\left(  \left\vert x^{\prime}-x\right\vert \right)
\frac{\left(  u_{j}\left(  x^{\prime}\right)  -u_{j}\left(  x\right)  \right)
}{\left\vert x^{\prime}-x\right\vert }%
\]
converges to the norm of the function%
\[
\Psi\left(  x^{\prime},x\right)  =H^{1/p}\left(  x^{\prime},x\right)
k_{\delta}^{1/p}\left(  \left\vert x^{\prime}-x\right\vert \right)
\frac{\left(  u\left(  x^{\prime}\right)  -u\left(  x\right)  \right)
}{\left\vert x^{\prime}-x\right\vert }.
\]
Since, additionally, up to a subsequence, $\left(  \Psi_{j}\right)  _{j}$
converges pointwise a.e. $\left(  x^{\prime},x\right)  \in\Omega_{\delta
}\times\Omega_{\delta}$ to $\Psi$, then $\Psi_{j}$ strongly converges to
$\Psi\left(  x^{\prime},x\right)  $ in $L^{p}\left(  \Omega_{\delta}%
\times\Omega_{\delta}\right)  $ (see \cite[Pag. 78]{Riesz}) and
(\ref{convergence_norms}) has been proved.
\end{proof}

\begin{remark}
\label{convergenceinnorm}The convergence (\ref{convergence_norms}), together
with the strong convergence of $\left(  u_{j}\right)  _{j},$ is precisely
equivalent to the strong convergence in $X$. In particular,
\begin{equation}
\lim_{j\rightarrow\infty}B\left(  u_{j}-u,u_{j}-u\right)
=0\label{convergence_norms_2}%
\end{equation}
and
\begin{equation}
\lim_{j\rightarrow\infty}B\left(  u_{j},u_{j}\right)  =B\left(  u,u\right)
.\label{convergence_norms_3}%
\end{equation}
We also realize that for any $h_{0}\in\mathcal{H}$ we have
\begin{align*}
\lim_{j\rightarrow\infty}B_{h}\left(  u_{j}-u,u_{j}-u\right)   &
=\lim_{j\rightarrow\infty}B_{\frac{h}{h_{0}}h_{0}}\left(  u_{j}-u,u_{j}%
-u\right) \\
& \geq\frac{h_{\min}}{h_{\max}}\lim_{j\rightarrow\infty}B_{h_{0}}\left(
u_{j}-u,u_{j}-u\right)
\end{align*}
Consequently, from (\ref{convergence_norms_main}) we deduce $\lim
_{j\rightarrow\infty}B_{h_{0}}\left(  u_{j}-u,u_{j}-u\right)  =0,$ for any
$h_{0}\in\mathcal{H}$ and thereby
\begin{equation}
\lim_{j\rightarrow\infty}B_{h_{0}}\left(  u_{j},u_{j}\right)  =\lim
_{j\rightarrow\infty}B_{h_{0}}\left(  u,u\right)
.\label{convergence_norms_3-h0}%
\end{equation}

\end{remark}

The convergences of the states we have just described above, are still valid
if we consider a sequence of sources $\left(  g_{j}\right)  _{j}$, uniformly
bounded in the dual space $X_{0}^{\prime}.$ Since $X_{0}$ is reflexive,
$X_{0}^{\prime}$ too, and we can ensure the sequence $\left(  g_{j}\right)
_{j}$ is weakly convergent, up to a subsequence, to an operator $g\in
X_{0}^{\prime}.$ Let $\left(  u_{j}\right)  _{j}$ and $u$ the underlying
states of $\left(  g_{j}\right)  _{j}$ and $g$ respectively. Then, thanks to
the precedent analysis, we know the sequence $\left(  u_{j}\right)  _{j},$ the
states associated to the controls $\left(  g_{j}\right)  _{j},$ converges
weakly to $u$ in $L^{p}\supset X,$ where $u$ is the stated associated to $g.$
Take now any element $L\in X_{0}^{\prime}.$ Under these circumstances there
exists a function $u_{L}\in u_{0}+X_{0}$ such that $B_{h}\left(
u_{L},w\right)  =\left\langle L,w\right\rangle _{X^{\prime}\times X}$ for any
$w\in X_{0}.$ Then, we easily deduce%
\[
\lim_{j\rightarrow\infty}\left\langle L,u_{j}-u_{0}\right\rangle
_{X_{0}^{\prime}\times X_{0}}=\left\langle L,u-u_{0}\right\rangle
_{X_{0}^{\prime}\times X_{0}}%
\]
The explanation of that relies on the strong convergence achieved in the above
corollary: indeed, H\"{o}lder's inequality and (\ref{convergence_norms_2})
straightforwardly provide
\begin{align*}
& \lim_{j\rightarrow\infty}\left\vert \left\langle L,u_{j}-u_{0}\right\rangle
_{X_{0}^{\prime}\times X_{0}}-\left\langle L,u-u_{0}\right\rangle
_{X_{0}^{\prime}\times X_{0}}\right\vert \\
& =\lim_{j\rightarrow\infty}\left\vert B_{h}\left(  u_{L},u_{j}-u\right)
\right\vert \\
& \leq\lim_{j\rightarrow\infty}B_{h}^{1/p^{\prime}}\left(  u_{L},u_{L}\right)
B_{h}^{1/p}\left(  u_{j}-u,u_{j}-u\right)  \\
& =0.
\end{align*}

The analysis performed explicitly confirm the fact that the sequence $\left(
u_{j}\right)  _{j}$ is weakly convergent to $u$ in $X_{0},$ or in other words,
the sequence of problems
\[
\min_{w\in u_{0}+X_{0}}\left\{  \frac{1}{p}B_{h}\left(  w,w\right)
-\int_{\Omega}g_{j}\left(  x\right)  w\left(  x\right)  dx\right\}
\]
$G$-converges to the problem
\[
\min_{w\in u_{0}+X_{0}}\left\{  \frac{1}{p}B_{h}\left(  w,w\right)
-\int_{\Omega}g\left(  x\right)  w\left(  x\right)  dx\right\}
\]
(see the abstract energy criterion established in \cite[Chapter 5, p.
162]{Jikov}).

\begin{theorem}
[Well posedness]\label{Well-posedness}There exists a solution $\left(
g,u\right)  $ to the control problem $\left(  \mathcal{P}^{\delta}\right)  $
given at (\ref{NLCP}).
\end{theorem}

\begin{proof}
Let $\left(  g_{j},u_{j}\right)  $ be a minimizing sequence. Then, up to
subsequence, we know that $g_{j}\rightharpoonup g$ weakly in $L^{p^{\prime}}$
and $u_{j}\rightarrow u$ strongly in $L^{p}.$ In addition, by Theorem
\ref{G-convergence} the couple $\left(  g,u\right)  $ is admissible for the
control problem, that is $\left(  g,u\right)  \in\mathcal{A}^{\delta}$.
Factually, this couple is a minimizer of the problem. To check that we observe
the infimum $i$ of the minimization principle can be computed as%
\[
i\doteq\lim_{j}I\left(  g_{j},u_{j}\right)  =\lim_{j\rightarrow\infty}%
\int_{\Omega}F\left(  x,u_{j}\left(  x\right)  ,g_{j}\left(  x\right)
\right)  dx.
\]
If we take into account the properties of $G$ is straightforward to verify the
that%
\begin{equation}
\lim_{j\rightarrow\infty}\int_{\Omega}G\left(  x,u_{j}\left(  x\right)
\right)  dx=\int_{\Omega}G\left(  x,u\left(  x\right)  \right)
dx\label{conver_G}%
\end{equation}
In addition, by using the Fatou's Lemma and the convergence
(\ref{convergence_norms_3-h0}) it is automatic to check that%
\begin{align*}
i &  \geq\text{ }\underset{j\rightarrow\infty}{\lim\inf}\int_{\Omega}F\left(
x,u_{j}\left(  x\right)  ,g_{j}\left(  x\right)  \right)  dx\\
&  \geq\int_{\Omega}F\left(  x,u\left(  x\right)  ,g\left(  x\right)  \right)
dx\\
&  =I\left(  g,u\right)  .
\end{align*}
The above inequality implies that $\left(  g,u\right)  $ is a minimizer.
\end{proof}

\begin{remark}
\label{unique}If $p=2$ and $G\left(  x,\cdot\right)  $ is convex, then the
solution of (\ref{NLCP}) is unique. The uniqueness is guaranteed because of to
the strict convexity of the function $t\rightarrow\left\vert t\right\vert
^{p^{\prime}}$ and the linearity of the state equation: if there are two
different solutions $\left(  g,u\right)  $ and $\left(  f,v\right)  ,$ then
the stated associated with the source $y_{s}\left(  x\right)  =sg\left(
x\right)  +\left(  1-s\right)  f\left(  x\right)  $ ($s\in\left(  0,1\right)
)$ is $u_{s}\left(  x\right)  =su\left(  x\right)  +\left(  1-s\right)
v\left(  x\right)  .$ If we apply the above properties of convexity, and the
one of the operator $B_{h_{0}}$ as well, then we arrive at
\begin{align*}
J_{\delta}\left(  y_{s},u_{s}\right)   & =\int_{\Omega}F\left(  x,u_{s}\left(
x\right)  ,y_{s}\left(  x\right)  \right)  dx=\int_{\Omega}\left(  G\left(
x,u_{s}\left(  x\right)  \right)  +\beta\left\vert y_{s}\left(  x\right)
\right\vert ^{p^{\prime}}\right)  dx+\gamma B_{h_{0}}\left(  u_{s}%
,u_{s}\right) \\
& <s\int_{\Omega}\left(  G\left(  x,u\left(  x\right)  \right)  +\beta
\left\vert g\left(  x\right)  \right\vert ^{p^{\prime}}\right)  dx+\left(
1-s\right)  \int_{\Omega}\left(  G\left(  x,v\left(  x\right)  \right)
+\beta\left\vert f\left(  x\right)  \right\vert ^{p^{\prime}}\right)  dx\\
& +\gamma sB_{h_{0}}\left(  u,u\right)  +\gamma\left(  1-s\right)  B_{h_{0}%
}\left(  v,v\right) \\
& =sJ_{\delta}\left(  g,u\right)  +\left(  1-s\right)  J_{\delta}\left(
f,v\right)  ,
\end{align*}
which is contradictory because both $\left(  g,u\right)  $ and $\left(
f,v\right)  $ are minimizers of $J_{\delta}.$
\end{remark}

\section{Convergence of the state equation if $\delta\rightarrow0$\label{S4}}

Assume the source $g$ and $u_{0}\in W^{1-1/p,p}\left(  \partial\Omega\right)
$ are fixed functions. If, for each $\delta,$ we consider the corresponding
sequence of states $\left(  u_{\delta}\right)  _{\delta}\subset u_{0}%
+W_{0}^{1,p}\left(  \Omega\right)  ,$ then%
\[
B_{h}\left(  u_{\delta},v\right)  =\int g\left(  x\right)  v\left(  x\right)
dx
\]
for any $v\in X_{0}.$ Consequently, as in the previous sections, we easily
prove $\left\Vert u_{\delta}\right\Vert _{L^{p}\left(  \Omega\right)  }$ and
$B_{h}\left(  u_{\delta},u_{\delta}\right)  $ are sequences uniformly bounded
in $\delta.$ Then, by using part \ref{S2_1_3} from Subsection \ref{S2_1},
these estimations imply the existence of a function $u^{\ast}\in u_{0}%
+W_{0}^{1,p}\left(  \Omega\right)  $ and a subsequence of $u_{\delta}$ (still
denoted $u_{\delta})$, $\ $such that $u_{\delta}\rightarrow u^{\ast}$ strongly
in $L^{p}\left(  \Omega\right)  $. Now, we look for the state equation that
should be satisfied by the pair $\left(  g,u^{\ast}\right)  $. The answer to
this question is given in the following convergence result:

\begin{theorem}
\label{Th2}%
\begin{equation}%
\begin{tabular}
[c]{l}%
$\displaystyle\lim_{\delta\rightarrow0}\min_{w\in u_{0}+X_{0}}\left\{
\frac{1}{p}B_{h}\left(  w,w\right)  -\int g\left(  x\right)  w\left(
x\right)  dx\right\}  \smallskip$\\
$\displaystyle=\min_{w\in u_{0}+W_{0}^{1,p}\left(  \Omega\right)  }\left\{
\frac{1}{p}\int_{\Omega}h\left(  x\right)  \left\vert \nabla w\left(
x\right)  \right\vert ^{p}dx-\int g\left(  x\right)  w\left(  x\right)
dx\right\}  $%
\end{tabular}
\ \ \ \ \ \label{7}%
\end{equation}
and $\left(  g,u^{\ast}\right)  \in\mathcal{A}^{loc}.$
\end{theorem}

\begin{proof}
We define the local state $u$ that corresponds to the source $g$ by means of
the local problem%
\[
\min_{w\in u_{0}+W_{0}^{1,p}\left(  \Omega\right)  }\left\{  \frac{1}{p}%
b_{h}\left(  w,w\right)  -\int g\left(  x\right)  w\left(  x\right)
dx\right\}
\]
If $u_{\delta}$ is the solution to the nonlocal state equation, then
$u_{\delta}$ solves%
\[
\min_{w\in u_{0}+X_{0}}\left\{  \frac{1}{p}B_{h}\left(  w,w\right)  -\int
g\left(  x\right)  w\left(  x\right)  dx\right\}
\]
and from $\left(  u_{\delta}\right)  _{\delta}$ we can extract a subsequence
strongly convergent to $u^{\ast}\in u_{0}+W_{0}^{1,p}\left(  \Omega\right)  $
in $L^{p}.$ By means of
\[
\lim_{\delta\rightarrow0}B_{h}\left(  u_{\delta},u_{\delta}\right)  \geq
\int_{\Omega}h\left(  x\right)  \left\vert \nabla u^{\ast}\left(  x\right)
\right\vert ^{2}dx,
\]
(see (\ref{Prel4})) we are allowed to write%
\begin{align*}
&  \lim_{\delta\rightarrow0}\min_{w\in u_{0}+X_{0}}\left\{  \frac{1}{p}%
B_{h}\left(  w,w\right)  -\int g\left(  x\right)  w\left(  x\right)
dx\right\}  \\
&  =\lim_{\delta\rightarrow0}\left(  \frac{1}{p}B_{h}\left(  u_{\delta
},u_{\delta}\right)  -\int g\left(  x\right)  u_{\delta}\left(  x\right)
dx\right)  \\
&  \geq\frac{1}{p}\int_{\Omega}h\left(  x\right)  \left\vert \nabla u^{\ast
}\left(  x\right)  \right\vert ^{p}dx-\int g\left(  x\right)  u^{\ast}\left(
x\right)  dx\\
&  \geq\frac{1}{p}\int_{\Omega}h\left(  x\right)  \left\vert \nabla u\left(
x\right)  \right\vert ^{p}dx-\int_{\Omega}g\left(  x\right)  u\left(
x\right)  dx\\
&  =\min_{w\in u_{0}+H_{0}^{1}\left(  \Omega\right)  }\left\{  \frac{1}{p}%
\int_{\Omega}h\left(  x\right)  \left\vert \nabla w\left(  x\right)
\right\vert ^{p}dx-\int g\left(  x\right)  w\left(  x\right)  dx\right\}  .
\end{align*}
We prove the reverse inequality: it suffices the usage of the convergence
given at (\ref{Prel5}) to realize that
\begin{align*}
&  \lim_{\delta\rightarrow0}\min_{w\in u_{0}+X_{0}}\left\{  \frac{1}{p}%
B_{h}\left(  w,w\right)  -\int g\left(  x\right)  w\left(  x\right)
dx\right\}  \\
&  =\lim_{\delta\rightarrow0}\left(  \frac{1}{p}B_{h}\left(  u_{\delta
},u_{\delta}\right)  -\int g\left(  x\right)  u_{\delta}\left(  x\right)
dx\right)  \\
&  \leq\lim_{\delta\rightarrow0}\left(  \frac{1}{p}B_{h}\left(  u,u\right)
-\int g\left(  x\right)  u\left(  x\right)  dx\right)  \\
&  =\frac{1}{p}b_{h}\left(  u,u\right)  -\int g\left(  x\right)  u\left(
x\right)  dx\\
&  =\min_{w\in u_{0}+H_{0}^{1}\left(  \Omega\right)  }\left\{  \frac{1}{p}%
\int_{\Omega}h\left(  x\right)  \left\vert \nabla w\left(  x\right)
\right\vert ^{p}dx-\int g\left(  x\right)  w\left(  x\right)  dx\right\}  .
\end{align*}
The above two estimations amount to state two consequences: on the one side,
these estimations clearly give the convergence result (\ref{7}). On the other
side, from the above discussion, it can be read that both $u$ and $u^{\ast}$
are solutions to the classical boundary problem (\ref{d}), and thus, by the
uniqueness proved for this problem, we deduce $u=u^{\ast}.$
\end{proof}

The proof we have just done provides the convergence of energies%
\begin{equation}
\lim_{\delta\rightarrow0}B_{h}\left(  u_{\delta},u_{\delta}\right)
=b_{h}\left(  u,u\right)  .\label{strong_conv_p2}%
\end{equation}
If we use (\ref{Prel5}) the above limit can be rewritten as follows:%
\begin{equation}
\lim_{\delta\rightarrow0}\int_{\Omega_{\delta}}\int_{\Omega_{\delta}}H\left(
x^{\prime},x\right)  k_{\delta}\left(  \left\vert x^{\prime}-x\right\vert
\right)  \left(  \frac{\left\vert u_{\delta}\left(  x^{\prime}\right)
-u_{\delta}\left(  x\right)  \right\vert ^{p}}{\left\vert x^{\prime
}-x\right\vert ^{p}}-\frac{\left\vert u\left(  x^{\prime}\right)  -u\left(
x\right)  \right\vert ^{p}}{\left\vert x^{\prime}-x\right\vert ^{p}}\right)
dx^{\prime}dx=0.\label{strong_conv_p2_bis}%
\end{equation}
Moreover, for the particular case $p=2$ we have strong convergence in $X_{0}:
$
\begin{equation}
\lim_{\delta\rightarrow0}B\left(  u_{\delta}-u,u_{\delta}-u\right)
=0.\label{linear-case}%
\end{equation}
(the proof is automatic). Furthermore, if $p=2$ \ and $h_{0}$ is any function
from $\mathcal{H},$ then
\[
\lim_{\delta\rightarrow0}B_{h_{0}}\left(  u_{\delta}-u,u_{\delta}-u\right)  =0
\]
and%
\begin{equation}
\lim_{\delta\rightarrow0}B_{h_{0}}\left(  u_{\delta},u_{\delta}\right)
=B_{h_{0}}\left(  u,u\right)  .\label{linear-case0}%
\end{equation}

\section{Approximation to the optimal control problem\label{S5}}

We know that, for each $\delta,$ there exists at least a solution $\left(
g_{\delta},u_{\delta}\right)  $ to the problem (\ref{NLCP}). Our purpose is
the asymptotic analysis of this sequence of solutions. We shall prove that the
limit in $\delta$ of the sequence $\left(  g_{\delta},u_{\delta}\right)  $,
the pair $\left(  g,u\right)  $ derived at the previous section, solves the
corresponding local optimal control problem $\left(  \mathcal{P}^{loc}\right)
$ defined in (\ref{e}). The tools we use are nothing more than those used in
Theorem \ref{Th2}.

\begin{theorem}
\label{Th3}Let $\left(  g_{\delta},u_{\delta}\right)  $ be the sequence of
solutions to the control problem (\ref{NLCP}) with $\gamma=0$. Then there
exists a pair $\left(  g,u\right)  \in L^{p^{\prime}}\left(  \Omega\right)
\times\left(  u_{0}+W_{0}^{1,p}\left(  \Omega\right)  \right)  $ and a
subsequence of indexes $\delta$ for which the following conditions hold:

\begin{enumerate}
\item $g_{\delta}\rightharpoonup g$ weakly in $L^{p^{\prime}}\left(
\Omega\right)  ,$ $u_{\delta}\rightarrow u$ strongly in $L^{p}\left(
\Omega\right)  $ as $\delta\rightarrow0.$

\item
\begin{equation}%
\begin{tabular}
[c]{l}%
$\displaystyle\lim_{\delta\rightarrow0}\min_{w\in u_{0}+X_{0}}\left\{
\frac{1}{p}B_{h}\left(  w,w\right)  -\int g_{\delta}\left(  x\right)  w\left(
x\right)  dx\right\}  \medskip$\\
$\displaystyle=\min_{w\in u_{0}+W_{0}^{1,p}\left(  \Omega\right)  }\left\{
\frac{1}{p}b_{h}\left(  w,w\right)  -\int g\left(  x\right)  w\left(
x\right)  dx\right\}  $%
\end{tabular}
\label{G-convergence-energy}%
\end{equation}
and $\left(  g,u\right)  \in\mathcal{A}^{loc}.$

\item $\left(  g,u\right)  $ is a solution to the local control problem
(\ref{e}).
\end{enumerate}
\end{theorem}

\begin{proof}
\begin{enumerate}
\item As in the previous analysis, it is clear that from any minimizing
sequence $\left(  g_{\delta},u_{\delta}\right)  $ we can extract a subsequence
of $\left(  g_{\delta}\right)  _{\delta}$ weakly convergent to $g$ in
$L^{p^{\prime}}$. Also, from the associated states $\left(  u_{\delta}\right)
_{\delta}$ we extract a subsequence that converges, strongly in $L^{p},$ to a
function $u^{\ast}\in u_{0}+W_{0}^{1,p}\left(  \Omega\right)  $. See
Subsection \ref{S2_1} part \ref{S2_1_3}.

\item We are going to see the state function $u^{\ast}$ is the one that
corresponds to the control $g:$\newline Let $u$ be the underlying state of
$g.$ Then, on the one side (\ref{Prel4}) allows us to write%
\begin{align*}
&  \lim_{\delta\rightarrow0}\min_{w\in u_{0}+X_{0}}\left\{  \frac{1}{p}%
B_{h}\left(  w,w\right)  -\int g_{\delta}\left(  x\right)  w\left(  x\right)
dx\right\} \\
&  =\lim_{\delta\rightarrow0}\left(  \frac{1}{p}B_{h}\left(  u_{\delta
},u_{\delta}\right)  -\int g_{\delta}\left(  x\right)  u_{\delta}\left(
x\right)  dx\right) \\
&  \geq\left(  \frac{1}{p}b_{h}\left(  u^{\ast},u^{\ast}\right)  -\int
g\left(  x\right)  u^{\ast}\left(  x\right)  dx\right) \\
&  \geq\min_{w\in u_{0}+W_{0}^{1,p}\left(  \Omega\right)  }\left\{  \frac
{1}{p}b_{h}\left(  w,w\right)  -\int g\left(  x\right)  w\left(  x\right)
dx\right\}  .
\end{align*}
On the other side, it is clear that (\ref{Prel5}) allows us to write
\begin{align*}
&  \lim_{\delta\rightarrow0}\min_{w\in u_{0}+X_{0}}\left\{  \frac{1}{p}%
B_{h}\left(  w,w\right)  -\int g_{\delta}\left(  x\right)  w\left(  x\right)
dx\right\} \\
&  \leq\lim_{\delta\rightarrow0}\left(  \frac{1}{p}B_{h}\left(  u,u\right)
-\int g_{\delta}\left(  x\right)  u\left(  x\right)  dx\right) \\
&  =\frac{1}{p}b_{h}\left(  u,u\right)  -\int g\left(  x\right)  u\left(
x\right)  dx\\
&  =\min_{w\in u_{0}+W_{0}^{1,p}\left(  \Omega\right)  }\left\{  \frac{1}%
{p}b_{h}\left(  w,w\right)  -\int g\left(  x\right)  w\left(  x\right)
dx\right\}  .
\end{align*}
From the above lines we infer that both $u$ and $u^{\ast}$ are solutions to
the local state problem (\ref{d}). Thereby, if we use uniqueness, which have
been checked at the end of Section \ref{S2}, we deduce that $u=u^{\ast}$ and
consequently, that $\left(  g,u\right)  \in\mathcal{A}^{loc}$. Furthermore,
another consequence we can derive is the following convergence of energies:
\[
\lim_{\delta\rightarrow0}B_{h}\left(  u_{\delta},u_{\delta}\right)
=b_{h}\left(  u,u\right)
\]

\item Take any $\left(  f,v\right)  \in\mathcal{A}^{loc}$ and consider the
sequence of solutions $\left(  f,v_{\delta}\right)  $ of the nonlocal boundary
problem $\left(  P^{\delta}\right)  $ with $g=f.$ Since $\left(  f,v_{\delta
}\right)  \in\mathcal{A}^{\delta}$ then
\begin{equation}
I\left(  f,v\right)  =\lim_{\delta}I_{\delta}\left(  f,v_{\delta}\right)
\geq\lim_{\delta}I_{\delta}\left(  g_{\delta},u_{\delta}\right)  \geq I\left(
g,u\right)  .\label{8}%
\end{equation}
To prove that we notice that the first equality of (\ref{8}) is true because
according to Theorem \ref{Th2} $v_{\delta}\rightarrow v\in u_{0}+W_{0}%
^{1,p}\left(  \Omega\right)  $ strongly in $L^{p}$ and $\left(  f,v\right)
\in\mathcal{A}^{loc}.$ If we use now the latter strong convergence and we pay
attention to the convergence (\ref{conver_G}), then it is straightforward to
deduce
\begin{align*}
\lim_{\delta}I_{\delta}\left(  f,v_{\delta}\right)   &  =\lim_{\delta}%
\int_{\Omega}\left(  G\left(  x,v_{\delta}\left(  x\right)  \right)
+\beta\left\vert f\left(  x\right)  \right\vert ^{p^{\prime}}\right)  dx\\
&  =\int_{\Omega}\left(  G\left(  x,v\left(  x\right)  \right)  +\beta
\left\vert f\left(  x\right)  \right\vert ^{p^{\prime}}\right)  dx\\
&  =I\left(  f,v\right)  .
\end{align*}
The first inequality of (\ref{8}) is due to the fact that $\left(  g_{\delta
},u_{\delta}\right)  $ is a sequence of minimizers for the cost $I_{\delta}$
and therefore,
\[
\lim_{\delta}I_{\delta}\left(  f,v_{\delta}\right)  \geq\lim_{\delta}%
I_{\delta}\left(  g_{\delta},u_{\delta}\right)  .
\]
And the second inequality of (\ref{8}) holds because (\ref{conver_G}) and
Fatou's Lemma yield%
\begin{align*}
\lim_{\delta}I_{\delta}\left(  g_{\delta},u_{\delta}\right)   &  \geq
\liminf_{\delta}\int_{\Omega}\left(  G\left(  x,u_{\delta}\left(  x\right)
\right)  +\beta\left\vert g_{\delta}\left(  x\right)  \right\vert ^{p^{\prime
}}\right)  dx\\
&  \geq\int_{\Omega}\left(  G\left(  x,u\left(  x\right)  \right)
+\beta\left\vert g\left(  x\right)  \right\vert ^{p^{\prime}}\right)  dx\\
&  =I\left(  g,u\right)  .
\end{align*}

\end{enumerate}
\end{proof}

\subsection{Case $p=2$}

The thesis of Theorem \ref{Th3} remains true when we put $\gamma>0$ and $p=2.$
To prove this statement we previously define the concrete optimization
problems we have to face.

The nonlocal control problem $\left(  \mathcal{P}^{\delta}\right)  $ read as
\begin{equation}
\min_{\left(  g,u\right)  \in\mathcal{A}^{\delta}}J_{\delta}\left(  g,u\right)
\label{NLCP_2}%
\end{equation}
where
\[
J_{\delta}\left(  g,u\right)  =I_{\delta}\left(  g,u\right)  +B_{h_{0}}\left(
u,u\right)  ,
\]
$B_{h_{0}}\left(  \cdot,\cdot\right)  $ is defined as in (\ref{bb}) with
$h=h_{0}$ and $p=2.$ The set of the admissibility is%
\[
\mathcal{A}^{\delta}=\left\{  \left(  f,v\right)  \in L^{2}\left(
\Omega\right)  \times X:v\text{ solves (\ref{b}) with }g=f\right\}  ,
\]
where (\ref{b}) has also to considered for the specific case $p=2.$It must be
underlined that for each $\delta$, there is a solution there is at least a
solution $\left(  g_{\delta},u_{\delta}\right)  \in L^{2}\left(
\Omega\right)  \times\left(  u_{0}+H_{0}^{1}\left(  \Omega\right)  \right)  $
to the problem (\ref{NLCP_2}). \smallskip

The corresponding local control problem $\left(  \mathcal{P}^{loc}\right)  $
is stated as
\begin{equation}
\min_{\left(  g,u\right)  \in\mathcal{A}^{loc}}J\left(  g,u\right) \label{e_2}%
\end{equation}
where%
\[
J\left(  g,u\right)  =I\left(  g,u\right)  +\gamma\int_{\Omega}h_{0}\left(
x\right)  \left\vert \nabla u\left(  x\right)  \right\vert ^{2}dx,
\]
with%
\[
\mathcal{A}^{loc}=\left\{  \left(  f,v\right)  \in L^{2}\left(  \Omega\right)
\times W^{1,2}\left(  \Omega\right)  :v\text{ solves (\ref{d}) with
}g=f\right\}
\]
and (\ref{d}) is assumed to be constrained to the case $p=2.$

\begin{theorem}
\label{Th4}Let $\left(  g_{\delta},u_{\delta}\right)  \ $be a sequence of
solutions to the problem (\ref{NLCP_2}). Then there exists a pair $\left(
g,u\right)  \in L^{2}\left(  \Omega\right)  \times\left(  u_{0}+W^{1,2}\left(
\Omega\right)  \right)  $ and a subsequence of indexes $\delta$ for which the
following conditions hold:

\begin{enumerate}
\item $g_{\delta}\rightharpoonup g$ weakly in $L^{2}\left(  \Omega\right)
\ $and $u_{\delta}\rightarrow u$ strongly in $L^{2}\left(  \Omega\right)  $ if
$\delta\rightarrow0.$

\item The identity (\ref{G-convergence-energy}) holds and $\left(  g,u\right)
\in\mathcal{A}^{loc}.$

\item $\left(  g,u\right)  $ is a solution to the local control problem
(\ref{e_2}).If in addition $G\left(  x,\cdot\right)  $ is assumed to be
convex, then the solution $\left(  g,u\right)  $ is unique.
\end{enumerate}
\end{theorem}

\begin{proof}
The procedure carried out for the proof of Theorem \ref{Th3} moves perfectly
into this context. It only remains to verify part 3. More concretely, we only
need to show%
\begin{equation}
J\left(  f,v\right)  \geq J\left(  g,u\right)  \text{ for any }\left(
f,v\right)  \in\mathcal{A}^{loc}.\label{ineq-fin}%
\end{equation}
To check (\ref{ineq-fin}) we take the sequence of solutions $\left(
f,v_{\delta}\right)  $ of the nonlocal boundary problem $\left(  P^{\delta
}\right)  $ with $g=f.$ By using the inequality%
\[
\lim_{\delta}B_{h_{0}}\left(  v,v\right)  =\int_{\Omega}h_{0}\left(  x\right)
\left\vert \nabla v\left(  x\right)  \right\vert ^{2}dx
\]
(see (\ref{Prel5})) and taking into account the limit (\ref{linear-case0}) we
realize that
\begin{align*}
J\left(  f,v\right)    & =\int_{\Omega}\left(  G\left(  x,v\left(  x\right)
\right)  +\beta\left\vert f\left(  x\right)  \right\vert ^{2}\right)
dx+\int_{\Omega}h_{0}\left(  x\right)  \left\vert \nabla v\left(  x\right)
\right\vert ^{2}dx\\
& =\int_{\Omega}\left(  G\left(  x,v\left(  x\right)  \right)  +\beta
\left\vert f\left(  x\right)  \right\vert ^{2}\right)  dx+\lim_{\delta
}B_{h_{0}}\left(  v,v\right)  \\
& =\lim_{\delta}\int_{\Omega}\left(  G\left(  x,v_{\delta}\left(  x\right)
\right)  +\beta\left\vert f\left(  x\right)  \right\vert ^{2}\right)
dx+\lim_{\delta}B_{h_{0}}\left(  v_{\delta},v_{\delta}\right)  \\
& =\lim_{\delta}J_{\delta}\left(  f,v_{\delta}\right)  .
\end{align*}
By applying now the optimality of $\left(  g_{\delta},u_{\delta}\right)  $ for
$J_{\delta},$ we clearly infer%
\[
J\left(  f,v\right)  =\lim_{\delta}J_{\delta}\left(  f,v_{\delta}\right)
\geq\lim_{\delta}J_{\delta}\left(  g_{\delta},u_{\delta}\right)  .
\]
And finally, by recalling the inequality
\[
\lim_{\delta}B_{h_{0}}\left(  u_{\delta},u_{\delta}\right)  \geq b_{h_{0}%
}\left(  u,u\right)
\]
(see \ref{Prel4}), we get
\begin{align*}
\lim_{\delta}J_{\delta}\left(  g_{\delta},u_{\delta}\right)    & =\lim
_{\delta}\int_{\Omega}\left(  G\left(  x,u_{\delta}\left(  x\right)  \right)
+\beta\left\vert g_{\delta}\left(  x\right)  \right\vert ^{2}\right)
dx+\lim_{\delta}B_{h_{0}}\left(  u_{\delta},u_{\delta}\right)  \\
& \geq\int_{\Omega}\left(  G\left(  x,u\left(  x\right)  \right)
+\beta\left\vert g\left(  x\right)  \right\vert ^{2}\right)  dx+b_{h_{0}%
}\left(  u,u\right)  \\
& =J\left(  g,u\right)  .
\end{align*}
By linking the above chain of inequalities we have proved (\ref{ineq-fin}).
Regarding uniqueness, it is sufficient to repeat the same argument of Remark
\ref{unique}.
\end{proof}

\begin{acknowledgement}
This work was supported by the Spanish Project MTM2017-87912-P, Ministerio de
Econom\'{\i}a, Industria y Competitividad (Spain), and by the Regional Project
SBPLY/17/180501/000452, JJ. CC. de Castilla-La Mancha. There are no conflicts
of interest to this work.
\end{acknowledgement}

\end{document}